\documentclass[11pt]{amsart}  
\usepackage{amsmath,amssymb,amsfonts,times,hyperref,xcolor}

\title{On the theta number of powers of cycle graphs}

\author{Christine Bachoc}
\address{Universit\'e Bordeaux I\\ Institut de Math\'ematiques de
  Bordeaux (IMB) \\ 351, cours de  la Lib\'eration, 33405 Talence,
  France}
\email{christine.bachoc@math.u-bordeaux1.fr}
\author{Arnaud P\^echer}
\address{Universit\'e Bordeaux I\\ Laboratoire Bordelais de Recherche Informatique (LaBRI) / INRIA Sud-Ouest\\ 351, cours de  la Lib\'eration, 33405 Talence, France}
\email{arnaud.pecher@labri.fr}
\thanks{Supported by the ANR/NSC project GraTel ANR-09-blan-0373-01, NSC98-2115-M-002-013-MY3 and NSC99-2923-M-110-001-MY3.}
\author{Alain Thi\'ery}
\address{Universit\'e Bordeaux I\\ Institut de Math\'ematiques de
  Bordeaux (IMB) \\ 351, cours de  la Lib\'eration, 33405 Talence,
  France}
\email{alain.thiery@math.u-bordeaux1.fr}


\date{\today}
\newtheorem{defi}{Definition}[section]
\newtheorem{definition}[defi]{Definition}
\newtheorem{proposition}[defi]{Proposition}
\newtheorem{theorem}[defi]{Theorem}

\newtheorem{remark}[defi]{Remark}
\newtheorem{corollary}[defi]{Corollary}
\newtheorem{lemma}[defi]{Lemma}
\newtheorem{claim}[defi]{Claim}

\newcommand{\N}{{\mathbb{N}}} 
\newcommand{\Z}{{\mathbb{Z}}} 
\newcommand{\Q}{{\mathbb{Q}}}
\newcommand{\R}{{\mathbb{R}}}

\newcommand{\OO}[1]{\displaystyle{\operatorname{O}\left(#1\right)}}

\newcommand{\Trace}{\operatorname{Trace}}
\newcommand{\Norm}{\operatorname{Norm}}

\newcommand{\Gal}{\operatorname{Gal}}

\begin{document}

\begin{abstract}
We give a closed formula for Lov\'asz's theta number of the powers of 
cycle graphs $C_k^d$ and of their complements, the circular complete 
graphs $K_{k/d}$. As a consequence, we establish that the circular
chromatic number of a circular  perfect graph is computable in
polynomial time. We also derive an asymptotic estimate for the theta
number
of $C_k^d$.
\end{abstract}

\maketitle

\section{Introduction}

Let $G=(V,E)$ be a finite graph with vertex set $V$ and edge set
$E$. The {\em clique number} $\omega(G)$ and the {\em chromatic number} $\chi(G)$
are classical invariants of $G$ which  can be defined in terms of graph homomorphisms. 

A {\em homomorphism} from a graph $G=(V,E)$ to a graph
$G'=(V',E')$ is a mapping $f: V \to V'$ which preserves adjacency: if
$ij$ is an edge of $G$ then $f(i)f(j)$ is an edge of $G'$.  If there
is a homomorphism from $G$ to $G'$, we write $G\to G'$. 
Then, the chromatic number of $G$ is the smallest number $k$ such that 
$G\to K_k$, where $K_k$ denotes the complete graph with $k$ vertices. 
Similarly, the clique number is the largest number $k$ such that 
$K_k \to G$. 

In the seminal paper \cite{Lovasz}, Lov\'asz introduced the so-called
{\em theta number} $\vartheta(G)$ of a graph. On one hand, this number
provides an approximation of $\omega(G)$ and of $\chi(G)$ since (this
is the celebrated {\em Sandwich Theorem}) 
\begin{equation*}
\omega(G)\leq \vartheta(\overline{G})\leq \chi(G),
\end{equation*}
where $\overline{G}$ stands for the complement of $G$.
On the other hand, this number is the optimal value of a {\em
  semidefinite program} \cite{VandenbergheBoyd}, and, as such, is computable in polynomial
time with polynomial space encoding accuracy \cite{VandenbergheBoyd}, \cite{GrotschelAl}.
In contrast, the computation of either of the clique number or the chromatic
number 
is known to be NP-hard.

By definition, a graph $G$ is {\em perfect}, if for every
induced subgraph $H$, $\omega(H)=\chi(H)$ \cite{Berge}. As
$\vartheta(\overline{G})=\chi(G)$ and $\chi(G)$ is an integer,
we have 

\begin{theorem}\label{GLS} (Gr\"otschel, Lov\'asz and Schrijver)
  \cite{GrotschelAl}
For every perfect graph, the chromatic number is computable in
polynomial time.
\end{theorem}

There are very few families of graphs for which an explicit formula for
the theta number is known. In \cite{Lovasz}, the theta numbers of the
cycles $C_k$ and of the Kneser graphs $K(n,r)$ are explicitly
computed.
In particular, it is shown that, if $k$ is an odd number,
\begin{equation}\label{theta cyclic}
\vartheta(C_{k}) = \frac{k \cos\left(\frac{\pi}{k}\right)}{1+\cos\left(\frac{\pi}{k}\right)}.
\end{equation}

In this paper, we give  a closed formula for the theta number of the
{\em circular complete graphs} $K_{k/d}$ and of their complements
$\overline{K_{k/d}}$. For $k\geq 2d$, the graph $K_{k/d}$ has $k$ vertices
$\{0,1,\dots,k-1\}$
and two vertices $i$ and $j$ are connected by an edge if 
$d\leq |i-j|\leq k-d$. We have $K_{k/1}=K_k$ and
$\overline{K_{k/2}}=C_k$.
More generally, the graph $\overline{K_{k/d}}$ is the $(d-1)^{\text{th}}$ power of 
the cycle graph $C_k$. Because the automorphism group of $K_{k/d}$ is
vertex transitive (it contains the cyclic permutation $(0,1,\dots,k-1)$), we have (see \cite[Theorem 8]{Lovasz})
\begin{equation}\label{theta prod}
\vartheta(K_{k/d})\vartheta(\overline{K_{k/d}})=k.
\end{equation}
Besides the case $d=2$ demonstrated by Lov\'asz, the only case
previously known was due to Brimkov et al. who proved 
in \cite{Brimkov} that, for $d=3$ and $k$ odd, 
\begin{equation}\label{theta Brimkov}
\vartheta\left(\overline{K_{k/3}}\right)  =   k \left( 1 - \frac{\frac{1}{2} - \cos\left(\frac{2\pi}{k}\lfloor \frac{k}{3} \rfloor\right) - \cos\left(\frac{2\pi}{k}\left(\lfloor \frac{k}{3} \rfloor+1\right)\right)}{\left(\cos\left(\frac{2\pi}{k}\lfloor \frac{k}{3} \rfloor\right)-1\right)\left(\cos\left(\frac{2\pi}{k}\left(\lfloor \frac{k}{3} \rfloor+1\right)\right)-1\right)} \right).
\end{equation}
In Section \ref{explicit formula}, we prove the following:
\begin{theorem}\label{Theorem 1}
Let $d\geq 2$, $k\geq 2d$, with $\gcd(k,d)=1$. Let, for $0\leq n\leq d-1$,
\begin{equation*}
c_n:=\cos\Big(\frac{2n\pi}{d}\Big), \quad
a_n:=\cos\Big(\Big\lfloor \frac{n k}{d} \Big\rfloor
\frac{2\pi}{k}\Big).
\end{equation*}
Then
\begin{equation}\label{formula for theta}
\vartheta(\overline{K_{k/d}})=\frac{k}{d}\sum_{n=0}^{d-1}\ \prod_{s=1}^{d-1}  \Big(\frac{c_n-a_s}{1-a_s}\Big).
\end{equation}
\end{theorem}

The notions of clique and chromatic numbers and of perfect
graphs have been refined  using circular complete graphs.
The {\em circular chromatic number} $\chi_c(G)$ of a graph $G$ was
first introduced by Vince in \cite{Vince}. It is the
minimum of the fractions $k/d$ for which $G\to K_{k/d}$. Later, Zhu defined the 
{\em circular clique number} $\omega_c(G)$ of $G$ to be the
maximum of the $k/d$ for which $K_{k/d}\to G$ and introduced the
notion of a {\em circular perfect graph}, a graph with the property that
every induced subgraph $H$ satisfies  $\omega_c(H)=\chi_c(H)$
(see Section 7 in \cite{Zhu2} for a survey on this notion). 
The class of circular perfect graphs extends in a natural way the
one  of perfect graphs. So one can ask for the properties of perfect
graphs that generalize to this larger class.
In this paper, we prove that Theorem \ref{GLS} still holds for
circular
perfect graphs:

\begin{theorem}\label{poly time}
For every circular perfect graph, the circular
chromatic number is computable in polynomial time.
\end{theorem}
In previous works, the polynomial time computability of the chromatic number
of circular perfect graphs was established in \cite{Pecher} and of the
circular chromatic number of strongly circular perfect graphs
(i.e. circular perfect graphs such that the complementary graphs are
also circular perfect) was proved in \cite{Pecher2}.

In contrary to perfect graphs, $\vartheta(\overline{G})$ does not give
directly the result, 
as $\vartheta(G)$ is not always sandwiched between $\omega_c(G)$ and
$\chi_c(G)$: for instance, $\vartheta( \overline{C_5} ) = \sqrt{5}$ 
and $\omega_c( \overline{C_5} )=\chi_c( \overline{C_5} )=5/2$. To
bypass this difficulty, we make use of this following basic
observation: 
by definition, for every graph $G$ with $n$ vertices such that
$\omega_c(G) = \chi_c(G) = k/d$, we have $\vartheta(\overline{G}) =
\vartheta(\overline{K_{k/d}})$, where $k,d \leq n$ (see the next
section for more details). 
Hence, to ensure the polynomial time computability of $\chi_c(G)$, it
is sufficient to prove that the values $\vartheta(\overline{K_{k/d}})$
with $k,d \leq n$ 
are all distinct and separated by at least $\epsilon$ for some $\epsilon$ with polynomial space encoding.

This paper is organized as follows: Section \ref{preliminaries} gathers the needed
definitions and properties of Lov\'asz theta number and of circular
numbers. Section \ref{explicit formula} proves
Theorem \ref{Theorem 1},
while Section \ref{separating} proves
Theorem \ref{poly time}. In Section \ref{asymptotic}, the asymptotic
estimate
$\displaystyle \vartheta(\overline{K_{k/d}}) =\frac{k}{d}+\OO{\frac{1}{k}}$ 
is obtained (Theorem \ref{Theorem 4}). 

\section{Preliminaries}\label{preliminaries}

The theta number
$\vartheta(G)$ of a graph $G=(V,E)$ was introduced in \cite{Lovasz}, where many equivalent
formulations are given. The one of \cite[Theorem 4]{Lovasz} has the
form of a semidefinite program:
\begin{equation}\label{theta}
\begin{array}{ll}
\vartheta(G)=\max \Big\{ \displaystyle \sum_{(x,y)\in V^2} B(x,y) &:\ B\in
\R^{V\times V},\ B\succeq 0,\\
&\displaystyle \sum_{x\in V} B(x,x)=1,\\
&B(x,y)=0\quad xy\in E\Big\}
\end{array}
\end{equation}
where $B\succeq 0$ stands for: $B$ is a symmetric, positive
semidefinite matrix. For a survey on semidefinite programming, we
refer to \cite{VandenbergheBoyd}. The dual program gives another
formulation for $\vartheta(G)$ (there is no duality gap here because
the identity matrix is a strictly feasible solution of \eqref{theta}
so the Slater condition is fulfilled):
\begin{equation}\label{theta dual}
\begin{array}{ll}
\vartheta(G)=\inf \Big\{ \displaystyle\  t &:\ B\in
\R^{V\times V},\ B\succeq 0,\\
&\displaystyle B(x,x) = t-1,\\
&B(x,y)=  -1\quad xy \notin E\Big\}
\end{array}
\end{equation}
From \eqref{theta dual} one can easily derive that, if $G\to G'$, then
$\vartheta(\overline{G})\leq \vartheta(\overline{G'})$. Indeed, if
$B'$ is an optimal solution of the dual program defining 
$\vartheta(\overline{G'})$, then the matrix $B$ defined by
$B(x,y):=B'(f(x),f(y))$ is feasible for $\vartheta(\overline{G})$.

The circular complete graphs $K_{k/d}$ have the property that
$K_{k/d}\to K_{k'/d'}$ if and only if $k/d\leq k'/d'$ (see
\cite{Bondy-Hell}). Thus the theta number
$\vartheta(\overline{K_{k/d}})$ only depends on the quotient $k/d$,
and we later conveniently assume that $k$ and $d$ are coprime.

From the definition, it follows that
\begin{equation*}
  \omega(G)\leq \omega_c(G)\leq \chi_c(G)\leq \chi(G).
\end{equation*}
Moreover, $\omega(G)=\lfloor \omega_c(G)\rfloor$,
$\chi(G)=\lceil \chi_c(G)\rceil$, and $\omega_c(G)$ and $\chi_c(G)$
are attained for pairs $(k,d)$ such that $k\geq 2d$ and $k\leq |V|$ (see \cite{Bondy-Hell}, \cite{Zhu}).

If $G$ is a  circular perfect graph, let $k,d$ be such that
$\gcd(k,d)=1$ and $\omega_c(G)=\chi_c(G)=k/d$. Because 
$G$ and $K_{k/d}$ are homomorphically equivalent,
$\vartheta(\overline{G})=\vartheta(\overline{K_{k/d}})$.
Summarizing, we have

\begin{proposition}\label{prop}
Let $G$ be a circular perfect graph with $n$ vertices. Then, 
\begin{enumerate}
\item[1.] $\omega_c(G)=\chi_c(G)=k/d$ for some $(k,d)$ such that $k\geq 2d$,
$k\leq n$, and $\gcd(k,d)=1$.  
\item[2.] $\vartheta(\overline{G})=\vartheta(\overline{K_{k/d}})$.
\end{enumerate}
\end{proposition}

\section{An explicit formula for the theta number of circular complete
  graphs}\label{explicit formula}
In this section we prove Theorem \ref{Theorem 1}. 
We start with an overview of our proof:
first of all we show that $\vartheta(\overline{K_{k/d}})$ is the
optimal value of a linear program (Proposition \ref{Proposition linear
program}). This step is a standard
simplification of a semidefinite program using its symmetries. In a
second step, a candidate for  an
optimal solution of the resulting  linear program is defined
(Definition \ref{Def opti}) as the unique solution
of a certain linear system. We give an interpretation of this element, in terms of
the coefficients of Lagrange interpolation polynomials on the basis of
Chebyshev polynomials. Then, playing with the dual linear program,
it is easy to prove that this element, {\em if feasible}, is
indeed optimal (Lemma \ref{Lemma opti}).
The last step, which is also the most technical, amounts to prove that
this elements is indeed feasible, i.e. essentially that its coordinates are  non negative
(Lemma \ref{Lemma non neg}). To that end, we  boil down to prove
that a certain polynomial $L_0(y)$  has non negative coefficients when expanded as a linear
combination of the Chebyshev polynomials (Lemma \ref{Lemma non neg bis}).

\subsection{A linear program defining the theta number}
The vertex set  of $G=\overline{K_{k/d}}$ can be identified with the
additive group $\Z/k\Z$, and the additive action of this group defines
automorphisms of this graph. This action allows to transform the
semidefinite program \eqref{theta} into a linear program, as follows:

\begin{proposition}\label{Proposition linear program}
Let $k_0:=\lfloor k/2\rfloor$. We have:
\begin{equation}\label{LP}
\begin{array}{ll}
\vartheta(\overline{K_{k/d}})=\max \Big\{ \  kf_0 \ :
\ & f_j \geq 0, \quad \displaystyle \sum_{j=0}^{k_0} f_j=1,\\
& \displaystyle \sum_{j=0}^{k_0} f_j\cos\Big(\frac{2j\ell\pi}{k}\Big)=0,\quad 1\leq
\ell\leq d-1 \ \Big\}\\
\end{array}
\end{equation}
and also:
\begin{equation}\label{DLP}
\begin{array}{ll}
\vartheta(\overline{K_{k/d}})=\min \Big\{ \  kg_0 \ : & \displaystyle \sum_{\ell=0}^{d-1} g_\ell\geq 1,\\
& \displaystyle \sum_{\ell=0}^{d-1} g_\ell\cos\Big(\frac{2\ell j\pi}{k}\Big)\geq 0,\quad 1\leq
j\leq k_0 \ \Big\}\\
\end{array}
\end{equation}
\end{proposition}

\begin{proof} Taking the average over the translations by the elements of
  $\Z/k\Z$, one constructs from a matrix  $B$ which is  optimal for
  \eqref{theta}, another optimal matrix which is translation
  invariant, i.e. which satisfies $B(x+z,y+z)=B(x,y)$ for all $x,y,z\in \Z/k\Z$. Thus
  one can restrict in \eqref{theta} to the matrices $B$ which are
  translation invariant. In other words, we can assume that
  $B(x,y)=F(x-y)$ for some $F:\Z/k\Z \to \R$. 
Then we can use the Fourier transform over $\Z/k\Z$ to express $F$ as 
\begin{equation*}
F(z)=\sum_{j=0}^{k-1} f_j e^{2ij z\pi/k}.
\end{equation*}
Then  $B\succeq 0$ if and only if $f_j=f_{k-j}$ and $f_j\geq 0$ for
all $j=0,\dots, k-1$. After a change from $f_j$ to $2f_j$ for $j\neq
0,k/2$, we can rewrite
\begin{equation*}
B(x,y)=\sum_{j=0}^{k_0} f_j \cos\Big(\frac{2j(x-y)\pi}{k}\Big).
\end{equation*}
Then, it remains to transfer to $(f_0,\dots,f_{k_0})$ the
constraints on $B$ that stand in \eqref{theta}. We have $\sum_{(x,y)\in (\Z/k\Z)^2} B(x,y)=k^2f_0$, and 
$\sum_{x\in \Z/k\Z} B(x,x)=k\sum_{j=0}^{k_0} f_j$. The edges
of $K_{k/d}$ are the pairs $(x,y)$ with $1\leq |x-y|\leq d-1$ so the
condition that $B(x,y)=0$ for all edges $(x,y)$ translates to 
\begin{equation*}
\sum_{j=0}^{k_0} f_j\cos\Big(\frac{2j\ell\pi}{k}\Big)=0,\quad 1\leq
\ell\leq d-1.
\end{equation*}
Changing $f_j$ to $f_j/k$ leads to \eqref{LP}. The linear program  \eqref{DLP} is  the dual
formulation of \eqref{LP}.
\end{proof}

\subsection{A candidate for an optimal solution of \eqref{LP}}
In order to understand the construction of this solution, it is worth
to take a look at the case when $d$ divides $k$. Indeed, in this case,
the  system of linear equations
\begin{equation*}
\sum_{j=0}^{k_0} f_j\cos\Big(\frac{2j\ell\pi}{k}\Big)=\delta_{0,l},\quad
0\leq \ell\leq d-1 
\end{equation*}
which is equivalent to
\begin{equation*}
\sum_{j=0}^{k-1} f_j'e^{2ij\ell\pi/k}=\delta_{0,l},\quad
0\leq \ell\leq d-1
\end{equation*}
where $f_j'=f_{k-j}'=f_j/2$ for $j\neq 0,k/2$, otherwise $f_j'=f_j$,
has an obvious solution $f=(f_0',\dots,f_{k-1}')$ defined as follows:
take $f_j=1/d$ for the indices $j$ which are multiples of $k/d$,
i.e. for $j=n k/d$, $n=0,\dots,d-1$. Take
$f_j=0$ for other indices. Then $f$ has exactly $d$ non zero
coefficients, is feasible because $f_j\geq 0$, and its objective value
equals $k/d$, which is also the optimal value of the linear program. 

In the case when $\gcd(k,d)=1$, none of the rational numbers  $nk/d$, for $n=0,\dots,d-1$,  are
integers.  Instead, we choose indices which are as close as possible,
namely we choose the indices of the form $\lfloor \frac{n k}{d}\rfloor$
for 
$0\leq n\leq d-1$ and set all other coefficients to zero. Then, the
$d$-tuple of coefficients which are not set to zero, satisfies  a 
linear system with  $d$ equations, and this linear system has a unique
solution. We shall prove
that in this way an optimal solution of \eqref{LP} is obtained.

Now we introduce some additional notations. The Chebyshev polynomials (\cite{Szego}), denoted
$(T_\ell)_{\ell\geq 0}$ are defined by the characteristic property:
$T_\ell(\cos(\theta))=\cos(\ell\theta)$. They can be iteratively computed by the
relation
$T_{\ell+1}(x)=2xT_\ell(x)-T_{\ell-1}(x)$ and the first terms $T_0=1$,
$T_1=x$. These polynomials are orthogonal  for the measure
$dx/\sqrt{1-x^2}$ supported on the interval $[-1,1]$. 

The numbers $a_n$, $0\leq n\leq d-1$, introduced in Theorem
\ref{Theorem 1}, come into play now. Recall that 
\begin{equation*}
a_n=\cos\Big(\Big\lfloor \frac{n k}{d} \Big\rfloor
\frac{2\pi}{k}\Big).
\end{equation*}
We remark that  the coefficients in the linear constraints of \eqref{LP}
associated  to the indices $\lfloor \frac{n k}{d}\rfloor$  are precisely equal to $T_{\ell}(a_n)$. 

We assume for the rest of this section that $\gcd(k,d)=1$. Then the real numbers $a_n$ are pairwise
distinct. We introduce the Lagrange polynomials (\cite{Szego})
associated to $(a_0,\dots, a_{d-1})$:
\begin{equation}\label{Lk}
L_n(y):= \prod_{\substack{s=0\\s\neq n}}^{d-1}\Big(\frac {y-a_s}{a_n-a_s}\Big).
\end{equation}
Now we have two basis for the space of polynomials of degree at most
equal to $d-1$: the Chebyshev basis $\{T_0,\dots,T_{d-1}\}$ and the
Lagrange basis $\{L_0,\dots,L_{d-1}\}$. We introduce the two $d\times d$ matrices
$T=(\tau_{\ell,n})$
and $L=(\lambda_{n,\ell})$ such that
\begin{equation}\label{def tau}
T_\ell(y)=\tau_{\ell,0}L_0(y)+\tau_{\ell,1}L_{1}(y)+\dots+\tau_{\ell,d-1}L_{d-1}(y)
\quad 0\leq \ell\leq d-1
\end{equation}
and 
\begin{equation}\label{def lambda}
L_n(y)=\lambda_{n,0}T_0(y)+\lambda_{n,1}T_{1}(y)+\dots+\lambda_{n,d-1}T_{d-1}(y)\quad
0\leq n\leq d-1.
\end{equation}
Obviously we have 
\begin{equation}
\tau_{\ell,n}=T_\ell(a_n)
\end{equation}
and
\begin{equation*}
TL=LT=I_d.
\end{equation*}

In particular, the $d$-tuple $(\lambda_{0,0},\lambda_{1,0},\dots ,
\lambda_{d-1,0})$ satisfies the equations:
\begin{equation}\label{eq lambda}
\sum_{n=0}^{d-1} \lambda_{n,0} T_{\ell}(a_n)  =\delta_{\ell,0},\quad
0\leq \ell\leq d-1.
\end{equation}

Now we can define our candidate for an optimal solution of \eqref{LP}:

\begin{definition}\label{Def opti}
With the above notations, let $f^*=(f_0^*,\dots, f_{k_0}^*)$ be defined by:
\begin{equation*}
\begin{cases}
f_j^*=\lambda_{n,0} &\text{if } j=\pm \big\lfloor \frac{n k}{d} \big\rfloor \mod k\\
f_j^*=0   & \text{otherwise}.
\end{cases}
\end{equation*}
\end{definition}

It remains to prove that $f^*$ is indeed optimal for \eqref{LP}. It
will result from the two following lemmas:

\begin{lemma}\label{Lemma non neg}
For all $j$, $0\leq j\leq k_0$,  $f_j^*\geq 0$.
\end{lemma}
We postpone the proof of Lemma \ref{Lemma non neg} to the next subsection.
\begin{lemma}\label{Lemma opti}
$f^*$ is an optimal solution of \eqref{LP}.
\end{lemma}
\begin{proof}
Lemma \ref{Lemma non neg}, joined with \eqref{eq lambda} shows that
$f^*$ is feasible. Thus we can derive the inequality:
\begin{equation*}
k\lambda_{0,0} \leq \vartheta(\overline{K_{k/d}}).
\end{equation*}
Now we claim that the element
$g^*=(\lambda_{0,0},\lambda_{0,1},\dots,\lambda_{0,d-1})$
is a feasible solution of the dual program \eqref{DLP}. 
For that we need to prove that
\begin{equation*}
\sum_{\ell=0}^{d-1} \lambda_{0,\ell}
\cos\Big(\frac{2\ell j\pi}{k}\Big)\geq \delta_{j,0},\quad
0\leq j\leq k_0
\end{equation*}
which can be rewritten as 
\begin{equation*}
\sum_{\ell=0}^{d-1} \lambda_{0,\ell}
T_{\ell}\Big(\cos\Big(\frac{2 j\pi}{k}\Big)\Big)\geq \delta_{j,0},\quad
0\leq j\leq k_0
\end{equation*}
or, taking account of \eqref{def lambda},
\begin{equation}\label{condition}
L_0\Big(\cos\Big(\frac{2 j\pi}{k}\Big)\Big)\geq \delta_{j,0},\quad
0\leq j\leq k_0.
\end{equation}
For $j=0$, \eqref{condition} holds because $L_0(1)=L_0(a_0)=1$. For $j\geq 1$, we take a
look at the position of
$\cos(2j\pi/k)$ with respect to the roots $a_1,\dots, a_{d-1}$ of
$L_0$.
Indeed, these roots  belong to the set
$\{\cos(2j\pi/k), j=1,\dots, k-1\}$, but it should be noticed that
they
 go in successive pairs. More precisely, $a_n$ and $a_{d-n}$ are equal
 to the first coordinate of  neighbor vertices of the regular
$k$-gone. So either 
$\cos(2j\pi/k)$ is equal to one of the $a_n$, or there is an even
number of roots $a_n$, $n\geq 1$, which are greater than
$\cos(2j\pi/k)$.  In the later case, $L_0(\cos(2j\pi/k))$ and $L_0(1)$
have the same sign. Since $L_0(1)=1$, we are done.

Since $g^*$ is a feasible solution of \eqref{DLP}, its objective
value, which is equal to $k\lambda_{0,0}$, upper  bounds
$\vartheta(\overline{K_{k/d}})$.
So we conclude that 
\begin{equation*}
\vartheta(\overline{K_{k/d}})=k\lambda_{0,0}
\end{equation*}
and that $f^*$ is an optimal solution of \eqref{LP}.
\end{proof}

\subsection{The proof of Lemma \ref{Lemma non neg}}
We want to prove that  $\lambda_{n,0}\geq 0$ for all $n=0,\dots,
d-1$. We first prove that this condition is equivalent to:
$\lambda_{0,\ell}\geq 0$ for all $0\leq \ell \leq d-1$.
\begin{lemma}\label{Lemma permut}
The tuples  $(\lambda_{n,0}, 0\leq n\leq d-1)$ and 
$(\lambda_{0,\ell}, 0\leq \ell \leq d-1)$ are equal up to a
permutation of their coordinates.
\end{lemma}
\begin{proof} 
It turns out that, up to a permutation of the $a_n$, the matrix $T$ is
symmetric. Since $\gcd(k,d)=1$, one can find  $v$, $1\leq v\leq d-1$, and
$t\geq 0$, such that $kv=1+td$. By definition,
\begin{equation*}
a_n:=\cos\Big(\Big\lfloor \frac{nk}{d} \Big\rfloor \frac{2\pi}{k}\Big)
\end{equation*}
only depends on $n \mod d$. Let us compute $a_{n'}$ where $n'=vn \mod d$.
Since $vnk/d=n/d+nt$, we have
\begin{equation*}
a_{n'}=\cos\Big(\Big\lfloor \frac{vnk}{d} \Big\rfloor \frac{2\pi}{k}\Big)
=\cos\Big( \frac{2nt\pi}{k} \Big).
\end{equation*}
If we set
\begin{equation}\label{x(k,d)}
x=x(k,d):=\cos\Big( \frac{2t\pi}{k}
\Big)=\cos\Big(\Big(\frac{v}{d}-\frac{1}{kd}\Big)2\pi\Big),
\end{equation}
we have
\begin{equation*}
a_{n'}=T_n(x).
\end{equation*}
If we reorder the $a_n$ according to the permutation
$n\mapsto vn \mod d$, which fixes $0$,
the coefficients of the corresponding  matrix $T$ are equal to:
\begin{equation*}
\tau_{\ell,n}=T_{\ell}(a_{n'})=T_\ell(T_n(x))=T_{\ell
  n}(x)=T_{n\ell}(x)=\tau_{n,\ell}.
\end{equation*}
Thus the new matrix $T$ is
symmetric. This reordering of the $a_n$,   permutes accordingly the
coordinates
of $(\lambda_{n,0}, 0\leq n\leq d-1)$. Also the matrix $L=T^{-1}$ has become
symmetric, so the permuted $\lambda_{n,0}$ are equal to
$\lambda_{0,n}$
(who have not changed in the procedure because the polynomial
$L_0(y)$ is not affected by the reordering of the $a_n$).
\end{proof}

The next lemma ends the proof of Lemma \ref{Lemma non neg}:

\begin{lemma}\label{Lemma non neg bis}
For all $0\leq \ell\leq d-1$, $\lambda_{0,\ell}\geq 0$.
\end{lemma}
\begin{proof}
Since $\prod_{s=1}^{d-1}(1-a_s)\geq 0$, we can replace $L_0(y)$ by 
\begin{equation*}
\prod_{s=1}^{d-1} (y-a_s)=\prod_{s=1}^{d-1} (y-T_s(x))
\end{equation*}
where $x$ is defined in  \eqref{x(k,d)}.
The right hand side becomes a polynomial in the variables $x$ and $y$,
depending only on $d$. This polynomial has an expansion in the
Chebyshev basis:
\begin{equation}\label{eq1}
\prod_{s=1}^{d-1} (y-T_s(x))=\sum_{\ell=0}^{d-1} Q_{\ell}(x)T_\ell(y).
\end{equation}
We introduce complex variables $X$ and $Y$, such that
$2x=X+1/X$ and $2y=Y+1/Y$.
Then, \eqref{eq1} becomes:
\begin{equation}\label{eq2}
\prod_{s=1}^{d-1} (Y-X^s)(Y-X^{-s})=Y^{d-1}\sum_{\ell=-(d-1)}^{d-1}
  Q_{\ell}'(x) Y^{\ell}
\end{equation}
where $Q'_0=2^{d-1}Q_0$ and, for $\ell=1,\dots, d-1$,
$Q'_{-\ell}=Q'_\ell=2^{d-2}Q_\ell$.
We want to prove that $Q'_{\ell}(x)\geq 0$ when $x$ is given by
\eqref{x(k,d)}.
To that end, we will prove that this sequence of numbers is decreasing:
\begin{equation}\label{eq3}
Q'_0(x)\geq Q'_1(x)\geq\dots \geq Q'_{d-1}(x)
\end{equation}
and since $Q'_{d-1}(x)=1$, we will be done. Now the idea is to
multiply the equation \eqref{eq2} by $(Y-1)$, so that the
successive differences $Q'_{\ell-1}(x)-Q'_{\ell}(x)$  appear in the
right hand side as the coefficients of $Y^\ell$.
We obtain,  setting $Q'_{-d}=Q'_d:=0$:
\begin{equation}\label{eq4}
\prod_{s=-(d-1)}^{d-1} (Y-X^s)=Y^{d-1} \sum_{\ell=-(d-1)}^{d} (Q'_{\ell-1}(x)- Q_{\ell}'(x))
Y^{\ell}.
\end{equation}
We let:
\begin{equation}\label{eq5}
P(Y):=\prod_{s=-(d-1)}^{d-1} (Y-X^s):=\sum_{j=0}^{2d-1} C_j(X)Y^j.
\end{equation}
We have:
\begin{align*}
P(XY) & = \prod_{j=-(d-1)}^{d-1}(XY-X^j)\\
 & = X^{2d-1}\prod_{j=-(d-1)}^{d-1}(Y-X^{j-1})\\
& = X^{2d-1}\frac{Y-X^{-d}}{Y-X^{d-1}} P(Y).
\end{align*}
This equation leads to:
\begin{equation*}
(Y-X^{d-1})\sum_{j=0}^{2d-1}
C_{j}(X)X^jY^{j}=(X^{2d-1}Y-X^{d-1})\sum_{j=0}^{2d-1} C_{j}(X) Y^{j}.
\end{equation*}
Comparing the coefficients of $Y^j$ in both sides, we obtain  the formula:
\begin{equation}\label{eq6} 
C_j(X) = C_{j-1}(X) \frac{X^{j-d}-X^d}{X^j-1}, \quad  1\leq  j\leq 2d-1.
\end{equation}
If  $X=e^{i\theta}$, we obtain in \eqref{eq6}
\begin{equation}\label{eq7}
 C_j(X) =
 C_{j-1}(X)\frac{\sin((\frac{j}{2}-d)\theta)}{\sin(\frac{j\theta}{2})},
 \quad 1\leq j\leq 2d-1.
\end{equation}
Thus, taking account of   \eqref{x(k,d)}, \eqref{eq4}, \eqref{eq5}
and \eqref{eq7}, the inequalities \eqref{eq3}
that we want to establish, are  equivalent to 
the non negativity of $\frac{\sin((\frac{j}{2}-d)\theta)}{\sin(\frac{j\theta}{2})}$ 
when $\theta = \big(\frac{v}{d}-\frac{1}{kd}\big) 2\pi$, $1\leq j\leq d-1$, and $k \geq 2d$.
Let us prove it now: let
\begin{equation*}
\begin{cases}
N=N(k,d):=\lfloor \frac{jv}{d}\rfloor\\
\varepsilon=\varepsilon(k,d) := \frac{jv}{d}-N.
\end{cases}
\end{equation*}
Since $v$ and  $d$ are coprime and $1\leq j < d$, we have
$\varepsilon \in \{\frac{1}{d},\frac{2}{d}, \ldots,
\frac{d-1}{d}\}$. We first study the sign of $\sin(\frac{j\theta}{2})$:
since $\frac{j\theta}{2} =\pi(N+\varepsilon-\frac{j}{kd})$, this
number belongs to $]N\pi,(N+1)\pi[$, which means that the
    sign of $\sin(\frac{j\theta}{2})$ is $(-1)^N$.

Now we determine the sign of $\sin((\frac{j}{2}-d)\theta)$ :
we have $(\frac{j}{2}-d)\theta
=\pi(N-2v+\varepsilon-\frac{j}{kd}+\frac{2}{k})$,
from which we obtain that $(\frac{j}{2}-d)\theta$ belongs to $](N-2v)\pi,(N-2v+1)\pi[$,
thus the sign of $\sin((\frac{j}{2}-d)\theta)$ equals $(-1)^{N+2v}$.

\end{proof}

\subsection{The end of the proof of Theorem \ref{Theorem 1}}

We have obtained an optimal  solution $f^*$  of \eqref{LP}, given in
Definition \ref{Def opti}, with objective value equal to
$k\lambda_{0,0}$. So we have 
\begin{equation}\label{eq8}
\vartheta(\overline{K_{k/d}})=k\lambda_{0,0}. 
\end{equation}
We recall that:
\begin{equation}\label{def lambda00}
L_0(y)=\lambda_{0,0}T_0(y)+\lambda_{0,1}T_{1}(y)+\dots+\lambda_{0,d-1}T_{d-1}(y).
\end{equation}
If we plug in \eqref{def lambda00} the value $y=c_n$ and sum up for
$n=0,\dots, d-1$, taking account of $T_0=1$ and 
$\sum_{n=0}^{d-1} T_j(c_n)=\sum_{n=0}^{d-1} \cos(2jn\pi/d)=0$, we
obtain the formula \eqref{formula for theta}.

\subsection{Other expressions for $\vartheta(\overline{K_{k/d}})$}

Alternatively, we can integrate \eqref{def lambda00} for the measure
$dy/\sqrt{1-y^2}$, for which the Chebyshev polynomials are orthogonal, leading to different expressions for 
$\vartheta(\overline{K_{k/d}})$:

\begin{theorem}\label{Theorem 2}
We have, with the notations of Theorem \ref{Theorem 1}:
\begin{align}\label{F1}
\vartheta(\overline{K_{k/d}})&= \frac{k}{\pi} \int_{-1}^1 L_0(y)
\frac{dy}{\sqrt{1-y^2}}\\\label{F2}
&= \frac{(-1)^{d-1}k}{\displaystyle \prod_{n=1}^{d-1}(1-a_n)} \sum_{j=0}^{\lfloor (d-1)/2
  \rfloor} \frac{1}{2^{2j}}\binom{2j}{j}
\sigma_{d-1-2j}(a_1,\dots,a_{d-1})
\end{align}
where $\sigma_0, \dots, \sigma_{d-1}$ denote the elementary symmetric
polynomials
in $d-1$ variables.
\end{theorem}

\begin{proof}
Integrating \eqref{def lambda00}  for the measure $dy/\sqrt{1-y^2}$ over the
interval $[-1,1]$ leads to \eqref{F1} because the Chebyshev
polynomials $T_n$ satisfy:
\begin{equation*}
\frac{1}{\pi}\int_{-1}^1 T_n(y) \frac{dy}{\sqrt{1-y^2}}= \delta_{n,0}.
\end{equation*}
Then, \eqref{F2} is obtained from \eqref{F1} with the monomial
expansion of  $L_0(y)$ and the formula (ref ??)
\begin{equation*}
\frac{1}{\pi}\int_{-1}^1 y^j \frac{dy}{\sqrt{1-y^2}}=
\begin{cases}
0 & \text{if } j \text{ is odd}\\
\frac{1}{2^j}\binom{j}{j/2} &\text{otherwise }.
\end{cases}
\end{equation*}
\end{proof}

\begin{remark} The expression
\eqref{formula for theta} specializes,
when $d=2$ and $d=3$, to the expressions given respectively in
\cite{Lovasz} and \cite{Brimkov}. Indeed, 
in the case $d=2$, we have $c_0=1$,
  $c_1=-1$, and $a_1=-\cos(\pi/k)$. Replacing in \eqref{formula for
    theta}, we recover \eqref{theta cyclic}.

For $d=3$, we have $c_1=c_2=-1/2$, $a_1=\cos(\lfloor
\frac{k}{3}\rfloor\frac{2\pi}{k})$ and $a_2=\cos((\lfloor
\frac{k}{3}\rfloor+1) \frac{2\pi}{k})$. We obtain in \eqref{formula for
    theta}
\begin{align*}
\vartheta(\overline{K_{k/3}})&=\frac{k}{3}\Big( 1+\frac{(c_1-a_1)(c_1-a_2)}{(1-a_1)(1-a_2)}+\frac{(c_2-a_1)(c_2-a_2)}{(1-a_1)(1-a_2)}\Big)\\
&=\frac{k(1/2+a_1a_2)}{(1-a_1)(1-a_2)}
\end{align*}
which agrees with the expression \eqref{theta Brimkov} given in \cite[Theorem 1]{Brimkov}.
\end{remark}

\section{Separating the $\vartheta(\overline{K_{k/d}})$}\label{separating}

In this section, we prove Theorem \ref{poly
  time}. Jointly with Proposition \ref{prop}, and following the discussion in
the Introduction, it will be an immediate consequence of
the following theorem:

\begin{theorem}\label{Theorem 3}
There exists an absolute and effective constant $c$ such that for all
$N \in \N, k \leq N, k' \leq N, k\geq 2d,  k'\geq 2d'$ with $\gcd(k,d)=\gcd(k',d')=1$, and $k/d \not = k'/d'$,
$$\vert\vartheta(\overline{K_{k/d}})-\vartheta(\overline{K_{k'/d'}})\vert \geq \frac{1}{c^{N^5}}.$$
\end{theorem}

We start with a proof of the weaker property that $\vartheta(\overline{K_{k/d}})\neq
\vartheta(\overline{K_{k'/d'}})$ if $k/d\not = k'/d'$.
\begin{theorem}\label{nonzero}
If $\vartheta(\overline{K_{k/d}}) = \vartheta(\overline{K_{k'/d'}})$
then $k/d = k'/d'$.
\end{theorem}
\begin{proof} Assume that 
$\vartheta(\overline{K_{k/d}}) = \vartheta(\overline{K_{k'/d'}})$ for
$k/d < k'/d'$. Since
 $\vartheta(\overline{K_{p/q}})$ is an increasing
function of $p/q$ (see Section \ref{preliminaries}), it implies that $\vartheta(\overline{K_{p/q}})$ is
constant for all $p/q \in [k/d,k'/d']$. This constant will be denoted
$\vartheta$ for simplicity.

\begin{claim}
The number $\vartheta$ is rational.
\end{claim}
\begin{proof}
Let $q \geq 5$ be a prime such that $1/q <
\frac{1}{4}(k'/d'-k/d)$. Then there exists $r$ such that
$r/q,(r+1)/q,(r+2)/q, (r+3)/q \in  [k/d,k'/d']$. Since $q \geq 5$, it
divides at most one of the four numbers $r,r+1,r+2,r+3$. Hence one can
find $p$ such that $p/q, (p+1)/q \in  [k/d,k'/d']$ and $q$ is prime to 
$p$ and $p+1$. 

For any positive integer $a$, denote $\zeta_a = \exp(2i\pi/a)$. 
We refer to \cite{Lang} for the basic notions of algebraic number
theory that will be involved next. For a number field $K$, we let
$\Gal(K)$ denote its Galois group over $\Q$. For number fields
$K\subset L$, and $x\in L$, $\Trace^L_K(x)$ and $\Norm^L_K(x)$ denote
respectively the trace and norm of $x$ in the extension $L/K$. 

It is well-known (see \cite{Lang}) that 
$$
\begin{array}{rcl}
\Psi_a\  :\ (\Z/a\Z)^\times & \longrightarrow &  \Gal(\Q(\zeta_{a}))\\
n & \mapsto & \sigma_n \text{ such that } \sigma_n(\zeta_{a}) =
\zeta_{a}^n
\end{array}
$$
is an isomorphism. Furthermore,
if $a$ and $b$ are coprime,
$$(\Z/ab\Z)^\times = (\Z/a\Z)^\times\times (\Z/b\Z)^\times$$
by Chinese Remainder Theorem.
It implies immediately that
$$\Gal(\Q(\zeta_{ab})) = \Gal(\Q(\zeta_{a}))\times \Gal(\Q(\zeta_{b})),$$
hence the fields $\Q(\zeta_{a})$ and $\Q(\zeta_{b})$ are linearly
disjoint over $\Q$.

We now compute $\vartheta = \vartheta(\overline{K_{p/q}})$ using
formula \eqref{formula for theta}. By definition, we have $c_n =\linebreak \cos(2n\pi/q)
= \frac12(\zeta_{q}^n+\zeta_{q}^{-n})= \sigma_n(c_1)$ for $1\leq n\leq
q-1$. It follows that
$$\vartheta =
\frac{p}{q}(1+\Trace_{\Q(\zeta_p)}^{\Q(\zeta_{pq})}(L_0(c_1))).$$
It gives immediately that $\vartheta \in \Q(\zeta_p)$. The same result
using $(p+1)/q$ leads to $\vartheta \in \Q(\zeta_{p+1})$. Since the
fields $\Q(\zeta_p)$ and $\Q(\zeta_{p+1})$ are linearly disjoint, this
proves the result.
\end{proof} 

\begin{claim}
The number $\vartheta$ is an integer.
\end{claim}

\begin{proof}
Let $\vartheta = \frac{a}{b}$ with $a,b \in \N$ coprime. Using the
same arguments as in the previous lemma, for any prime $p$
such that $1/p < \frac{1}{4}(d/k-d'/k')$, one can find $q$, with $p$
coprime to $q$ and $q+1$, such that
$q/p,(q+1)/p \in [d/k, d'/k']$ . It means that $p/q, p/(q+1) \in
[k/d,k'/d']$.

Using formula \eqref{formula for theta} for $p/q$, one sees that
$x = q\prod_{n=1}^{q-1}(2-2a_n)\vartheta$ is an algebraic integer,
hence $\Norm_{\Q}^{\Q(\zeta{pq})}(x) \in \Z$.
We now compute this norm. 

Since $q\vartheta$ is rational,
$\Norm_{\Q}^{\Q(\zeta{pq})}(q\vartheta) = (q\vartheta)^{\phi(pq)}$
where $\phi$ is the Euler function.
Since $p$ is a prime, $a_n$ is a conjugate of $a_1$ for all $1 \leq n\leq q-1$, hence 
$\Norm_{\Q}^{\Q(\zeta{pq})}(\prod_{n=1}^{q-1}(2-2a_n)) = (\Norm_{\Q}^{\Q(\zeta{pq})}(2-2a_1))^{q-1}$.
We also have
$$2-2a_1= 2 - 2\cos(\lfloor \frac{p}{q}\rfloor\frac{2\pi}{p}) =
(1-\zeta_p^{\lfloor \frac{p}{q}\rfloor})(1-\zeta_p^{-\lfloor
  \frac{p}{q}\rfloor}).$$
 Hence 
$\Norm_{\Q}^{\Q(\zeta{pq})}(2-2a_1) = (\Norm_{\Q}^{\Q(\zeta{pq})}(1-\zeta_p))^2$. Finally,
\begin{align*}
\Norm_{\Q}^{\Q(\zeta{pq})}(1-\zeta_p) &= \Norm_{\Q}^{\Q(\zeta{p})}(\Norm_{\Q(\zeta_p)}^{\Q(\zeta{pq})}(1-\zeta_p)) \\
&=  (\Norm_{\Q}^{\Q(\zeta{p})}(1-\zeta_p))^{\phi(q)} \\
& = p^{\phi(q)}
\end{align*}
(see \cite{Lang}). Summing up all partial results, one gets
$$(q\frac{a}{b})^{\phi(pq)}p^{2(q-1)\phi(q)} \in \Z.$$
If $l\not = p$ is a prime factor of $b$, then $l$ divides $q$ by the previous formula. But the same formula holds with $q+1$, hence $l$ divides also $q+1$. It follows that $b$ is a power of $p$. But this is true for any $p$ large enough. Hence $b = 1$. This proves the result.
\end{proof}

To finish the proof of Theorem \ref{nonzero},  we use the following result from
\cite{Pecher} (see also \cite{GrotschelAl}): 
if $\vartheta(\overline{K_{k/d}}) \in \N$ then $k/d
\in \N$. But every rational number in the interval
$[k/d,k'/d']$ cannot be an integer.
\end{proof} 

We can now start the proof of Theorem \ref{Theorem 3}. It is based on the following obvious lemma.

\begin{lemma}\label{minoration-entier}
Let $\alpha$ be a non zero algebraic integer of degree less than $\delta$ and $c\geq 1$ such that the absolute values of the conjugates of $\alpha$ are less than $c$ then
$$\vert\alpha\vert \geq \frac{1}{c^\delta}$$
\end{lemma}
\begin{proof}
Since $\alpha$ is a non zero algebraic integer , $ \vert \Norm_{\Q}^{\Q(\alpha)}(\alpha)\vert\geq 1$. It follows immediately that
$$\vert\alpha\vert c^{\delta-1} \geq 1.$$
\end{proof}

Let 
\begin{align*}
\alpha &= dd'\prod_{n=1}^{d-1}(2-2a_n)\prod_{n=1}^{d'-1}(2-2a'_n)(\vartheta(\overline{K_{k/d}})-\vartheta(\overline{K_{k'/d'}}))\\
&= kd'\prod_{n=1}^{d'-1}(2-2a'_n)\sum_{n=0}^{d-1}\prod_{m=1}^{d-1}(2c_n-2a_m)
-k'd\prod_{n=1}^{d-1}(2-2a_n)\sum_{n=0}^{d'-1}\prod_{m=1}^{d'-1}(2c'_n -2a'_m)
\end{align*}
with the obvious notations $c'_n:=\cos\Big(\frac{2n\pi}{d'}\Big)$ and
$a'_n:=\cos\Big(\Big\lfloor \frac{n k'}{d'} \Big\rfloor \frac{2\pi}{k'}\Big)$.
The number $\alpha$ is thus an algebraic integer, and 
it is non zero by Theorem \ref{nonzero}. Moreover it belongs to $\Q(\zeta_{kdk'd'})$, hence its degree is less than $N^4$.

Let $\beta$ be a conjugate of $\alpha$. Since the absolute values of the conjugates of $a_n, a'_n, c_n$ and $c'_n$ are all less than 1, one gets
$$\vert\beta\vert \leq kd'4^{d'-1}d4^{d-1}+k'd4^{d-1}d'4^{d'-1}\leq 2N\frac{N}{2}4^{\frac{N}{2}}\frac{N}{2}4^{\frac{N}{2}} \leq N^34^N.$$
It follows from Lemma \ref{minoration-entier} that
$$\vert\alpha\vert \geq \frac{1}{(N^34^N)^{N^4}}.$$
Furthermore, $\vert
dd'\prod_{n=1}^{d-1}(2-2a_n)\prod_{n=1}^{d'-1}(2-2a'_n)\vert 
\leq N^24^N.$ This implies immediately that
$$\vert\vartheta(\overline{K_{k/d}})-\vartheta(\overline{K_{k'/d'}})\vert 
\geq  \frac{1}{N^24^N(N^34^{N})^{N^4}}.$$
This finishes the proof of Theorem \ref{Theorem 3}.

\section{The asymptotic behaviour of $\vartheta(\overline{K_{k/d}})$}\label{asymptotic}

From Lov\'asz's formula \eqref{theta cyclic},  the asymptotic behaviour of  the theta number of odd holes $C_{2k+1}$ is (see \cite{Bohman} for instance):
\begin{equation} \label{eq_oddholes}
 \vartheta\left( C_{2k+1} \right)  =  \frac{2k+1}{2} + O\left( \frac{1}{k} \right).
\end{equation}
In general, we have $\vartheta(\overline{K_{k/d}})\leq k/d$. Indeed, 
\begin{equation*}
\vartheta(\overline{K_{k/d}})=k/\vartheta(K_{k/d})\leq
k/\omega(\overline{K_{k/d}})=k/d.
\end{equation*}
In this section, we prove:

\begin{theorem}\label{Theorem 4}
 If $d \geq 3$ and $k \geq 4d^3/\pi$ then $\vartheta(\overline{K_{k/d}}) \geq \frac{k}{d} - \frac{4e\pi^2}{3}\frac{d}{k}$. 
 Hence, for $d$ fixed, 
\begin{equation*}
\vartheta(\overline{K_{k/d}}) = \frac{k}{d}+\OO{\frac{1}{k}}.
\end{equation*}
\end{theorem}

Notice that for $d=2$, Theorem \ref{Theorem 4} agrees with Equation \eqref{eq_oddholes}.

\begin{proof}
Let $d \geq 3$ and $k \geq 4d^3/\pi$.
For every $0 \leq i \leq d-1$, let $c_i = \cos(2i\pi/d)$, $\sigma_i =
\sin(2i\pi/d)$, 
$a_i=\cos\Big(\Big\lfloor \frac{i k}{d} \Big\rfloor
\frac{2\pi}{k}\Big)$ and $\delta_i = c_i - a_i$. 
We have  $a_i = \cos(2i\pi/d - 2\pi \epsilon_i)$, with  $\epsilon_i = s_i/kd$ and $s_i = ik \mod d$.

\begin{claim}
For every  $1\leq j \leq d-1$, we have
\begin{align}
\label{eq_productci-cj}
\prod_{\substack{ i=0 \\ i \neq j, d-j}}^{d-1} \left( c_j - c_i \right) &=  \frac{-d^2}{2^d \sigma_j^2} \quad\text{ if } j \neq 0, d/2 \\
\label{eq_productci-cjb}
\prod_{\substack{ i=0 \\ i \neq j, d-j}}^{d-1} \left( c_j - c_i \right) &=  \frac{-d^2}{2^{d-1}} \quad\text{ if } j=d/2 \\
\label{eq_sum1/1-ci}
\sum_{i=1}^{d-1} \frac{1}{1-c_i}  &=  \frac{d^2-1}{6}
\end{align}
\end{claim}

\begin{proof}
The proof of these equalities is a short computation and the details
are omitted. Equation \eqref{eq_productci-cjb} (respectively
\eqref{eq_productci-cj}, \eqref{eq_sum1/1-ci})
is obtained by taking the first derivative (respectively the second derivative, the
third derivative) with respect to $x$ of the equality $T_k(\cos x)= \cos(kx)$,
taking into account the identity  $T_k(x)-1 = 2^{k-1}
\prod_{i=0}^{k-1}(x-c_i)$, then by evaluating the resulting identity
at  $\pi$ (respectively $2j\pi/q$, respectively $0$).
\end{proof}

\begin{claim}
For every  $1\leq j \leq d-1$, we have
\begin{align}
\label{eq_deltajdeltaq-j}
 \left| \frac{\delta_j \delta_{d-j}}{\sigma_j^2} \right| &\leq \frac{4\pi^2}{k^2} \left(1+\frac{d\pi}{k}\right) \quad\text{ if } j \neq {d/2}\\
\label{eq_deltaq/2}
\left| \delta_j \right| &\leq \frac{4\pi^2}{k^2} \quad\text{ if } j = {d/2}
\end{align}
\end{claim}

\begin{proof}
For every $1\leq j \leq d-j$, let $z_j \in [2j\pi/d - 2\pi \epsilon_j,
  2j\pi/d]$ such that $\delta_j = - 2 \pi \epsilon_j \sin(z_j)$. As
$|\sin(z_j)| \leq |\sigma_j|$ or  $|\sin(z_{d-j})| \leq |\sigma_j|$,
we get 
\begin{equation*}
|\delta_j \delta_{d-j}| \leq 4\pi^2 |\sigma_j| (|\sigma_j| +
2\pi/k)/k^2.
\end{equation*}
Taking into account  $|\sigma_j| \geq \sin\frac{\pi}{d} \geq 2/d$, we
obtain \eqref{eq_deltajdeltaq-j}. 

The inequality \eqref{eq_deltaq/2} is straightforward.
\end{proof}

\begin{claim}
\label{cl_num}
 For every $1\leq j \leq d-1$, we have
\begin{equation} \label{eq_cj-ai}
\prod_{i=1}^{d-1} |c_j-a_i|  \leq  \frac{e \pi^2}{2^{d-3} (1-c_j)} \frac{d^2}{k^2}.
\end{equation}
\end{claim}

\begin{proof}
Let $m_j=\delta_j \delta_{d-j}$ if $j \neq d/2$, and $m_{d/2}=\delta_{d/2}$.
We have 
\begin{align*}
\prod_{i=1}^{d-1} |c_j-a_i| &=  \left| m_j \prod_{ \substack{i=1 \\ i  \neq j, d-j} }^{d-1} (c_j-c_i) \right| \left| \prod_{\substack{i=1 \\ i  \neq j, d-j}}^{d-1}  \left( 1 + \frac{\delta_i}{c_j - c_i} \right) \right| \\
 &\leq  \frac{\pi^2}{2^{d-3} (1-c_j)} \frac{d^2}{k^2}   \left| \prod_{\substack{ i=1 \\ i \neq j, d-j}}^{d-1}  \left( 1 + \frac{\delta_i}{c_j - c_i} \right) \right|  \text{ due to \eqref{eq_productci-cj}, \eqref{eq_productci-cjb}, \eqref{eq_deltajdeltaq-j}, \eqref{eq_deltaq/2}} \\
 &\leq \frac{\pi^2}{2^{d-3} (1-c_j)} \frac{d^2}{k^2}  \exp\left( \frac{2\pi}{k}  \sum_{\substack{ i=1 \\ i \neq j, d-j}}^{d-1} \frac{1}{\left| c_j - c_i \right|} \right) \\
&\leq  \frac{\pi^2}{2^{d-3} (1-c_j)} \frac{d^2}{k^2}  \exp\left( \frac{4}{\pi} \frac{d^3}{k}  \right) \text{ since } \left| c_j - c_i \right| \geq \frac{\pi^2}{2d^2} \text{ for every } i \neq j, d-j \\
&\leq  \frac{\pi^2 e}{2^{d-3} (1-c_j)} \frac{d^2}{k^2} \quad\text{ as }k \geq 4d^3/\pi.
\end{align*}
\end{proof}

\begin{claim}
\label{cl_den}
We have
\begin{equation*}
\prod_{i=1}^{d-1} \left( 1-a_i \right)  \geq  \frac{d^2}{2^d}.
\end{equation*}
\end{claim}

\begin{proof}
Indeed, 
\begin{align*}
\prod_{i=1}^{d-1} \left( 1-a_i \right) &= \prod_{ i=1  }^{d-1} (1-c_i) \prod_{i=1}^{d-1}  \left( 1 + \frac{\delta_i}{1 - c_i} \right) \\
		&\geq \frac{d^2}{2^{d-1}} \prod_{i=1}^{d-1}  \left( 1 - \frac{2 \pi}{k(1 - c_i)} \right) \text{ due to \eqref{eq_productci-cjb}}.\\
\end{align*}

If $x_1, \ldots, x_l$ are real numbers belonging to $[0,1]$, then $\prod_{i=1}^{l} \left( 1-x_i \right) \geq 1-(x_1+ \ldots +x_l)$. Since for every $i$, $\frac{2 \pi}{k(1 - c_i)} \leq 1$, it follows:
\begin{align*}
\prod_{i=1}^{d-1} \left( 1-a_i \right) &\geq \frac{d^2}{2^{d-1}} \left( 1-\frac{2\pi}{k}\sum_{i=1}^{d-1}\frac{1}{1-c_i}\right) \\
	 &\geq \frac{d^2}{2^d} \quad \text{ due to \eqref{eq_sum1/1-ci} and }k \geq 4d^3/\pi.
\end{align*}
\end{proof}

Now we are ready to prove Theorem \ref{Theorem 4}. We have the following chain of inequalities:
\begin{align*}
\vartheta\left( \overline{K_{k/d}} \right) &= \frac{k}{d} +
\frac{k}{d} \sum_{n=1}^{d-1} \frac{\prod_{i=1}^{d-1} \left( c_n-a_i
  \right)}{\prod_{i=1}^{d-1} \left( 1-a_i \right)} \quad\text{ from
  \eqref{formula for theta}}\\
	&\geq \frac{k}{d} - \frac{k}{d} \frac{2^d}{d^2} \sum_{n=1}^{d-1} \left|\prod_{i=1}^{d-1} \left( c_n-a_i \right)\right| \text{ by Claim \ref{cl_den}} \\
	&\geq \frac{k}{d} - \frac{k}{d} \frac{2^d}{d^2} \frac{e \pi^2}{2^{d-3}} \frac{d^2}{k^2} \sum_{n=1}^{d-1} \frac{1}{1-c_n} \text{ by Claim \ref{cl_num}} \\
	&\geq \frac{k}{d} - \frac{4e\pi^2}{3}\frac{d}{k}   \quad\text{ due to \eqref{eq_sum1/1-ci}} 
 \end{align*}

\end{proof}

Notice that Theorem \ref{Theorem 4} shows that $\vartheta$ is close to
the circular chromatic number of {\em dense} circular perfect graphs (where dense means that the clique number is large compared to the stability number):
\begin{corollary}
For every $\epsilon >0$, for every positive integer $\alpha$, there is a positive integer $\omega$ such that for every circular-perfect graph $G$ satisfying $\omega(G) \geq \omega$ and $\alpha(G) \leq \alpha$, we have $|\vartheta\left( \overline{G} \right) - \chi_c(G) | \leq \epsilon$.
\end{corollary}

\end{document}